\documentclass[12pt, reqno]{amsart}
\usepackage[a4paper,left=3cm,right=3cm,top=3cm,bottom=3cm]{geometry} 
\usepackage[utf8]{inputenc}
\usepackage[T1]{fontenc}
\usepackage[english]{babel}
\usepackage[centertags]{amsmath}
\usepackage{amstext,amssymb,amsopn,amsthm, amsfonts}
\usepackage{nicefrac, esint}
\usepackage{mathrsfs}
\usepackage{dsfont}
\usepackage{bbm}
\usepackage{thmtools}
\usepackage[colorinlistoftodos,prependcaption,color=white,bordercolor=blue]{todonotes}

\theoremstyle{plain}
\newtheorem{thm}{Theorem}[section]
\newtheorem{lem}[thm]{Lemma}
\newtheorem{prop}[thm]{Proposition}
\newtheorem{cor}[thm]{Corollary}

\theoremstyle{definition}
\newtheorem{defn}[thm]{Definition}
\newtheorem{example}[thm]{Example}

\newtheorem*{rem*}{Remark}

\newcommand{\Rd}{{\mathbb{R}^d}}

\newcommand{\Rtd}{\mathbb R^d \times \mathbb R^d}
\newcommand{\N}{{\mathbb{N}}}
\newcommand{\Sd}{{\mathbb{S}^{d-1}}}
\newcommand{\R}{\mathbb{R}}
\newcommand{\Nop}{\mathcal N_\varepsilon}
\newcommand{\divop}{\mathcal D_\varepsilon}

\newcommand{\aeps}{\alpha_\eps}

\newcommand{\eps}{\varepsilon}

\newcommand{\intd}{\mathrm{d}}
\newcommand{\nv}{\vec{n}}

\newcommand{\niceQed}{\hfill$\qed$}

\newcommand{\m}{\mu}
\newcommand{\meps}{{\m}_\eps}

\makeatletter
\newcommand*{\defeq}{\mathrel{\rlap{\raisebox{0.3ex}{$\m@th\cdot$}} \raisebox{-0.3ex}{$\m@th\cdot$}} =}
\makeatother

\DeclareMathOperator{\dist}{dist}
\DeclareMathOperator{\diver}{div}

\allowdisplaybreaks
\numberwithin{equation}{section}
\swapnumbers

\usepackage{csquotes}
\usepackage{dsfont}

\usepackage{hyperref}
\hypersetup{
		colorlinks=true, 
		linktoc=all,     
		linkcolor=blue, 
		citecolor=black, 
	}

\addto\extrasenglish{%
}

\begin{document}
\begin{abstract}
We present a new proof of the classical divergence theorem in bounded domains. Our proof is based on a nonlocal analog of the divergence theorem and a rescaling argument. Main ingredients in the proof are nonlocal versions of the divergence and the normal derivative. We employ these to provide definitions of well-known nonlocal concepts such as the fractional perimeter.
\end{abstract}

\author{Solveig Hepp}
\address{Fakultät für Mathematik, Universität Bielefeld, Germany}
\email{shepp@math.uni-bielefeld.de}

\author{Moritz Kassmann}
\address{Fakultät für Mathematik, Universität Bielefeld, Germany}
\email{moritz.kassmann@uni-bielefeld.de}

\thanks{Solveig Hepp gratefully acknowledges financial support by the German Research Foundation (GRK 2235 - 282638148). }

\title{The divergence theorem and nonlocal counterparts}

\maketitle

\section{Introduction}\label{sec:intro}
The divergence theorem is indisputably one of the most significant theorems in analysis. 
Its history is closely linked with the names of Lagrange,  Gauss, Green, Ostrogradsky, and Stokes. 
In its standard version, the theorem states that
\begin{align}\label{eq:div-theo}\tag{DT}
    \int_\Omega \diver F(x) \intd x = \int_{\partial\Omega} F(x) \cdot \nv(x) \intd \sigma_{d-1}(x),
\end{align}
for a bounded $C^1$-domain $\Omega \subset \Rd$ and a continuously differentiable vector field $F: \overline\Omega \to \Rd$. In this notation, $\nv(x)$ is the outward unit normal vector at a point $x$ on the boundary of $\Omega$ and $\sigma_{d-1}$ is the $(d-1)$-dimensional surface measure. The divergence theorem has been established in different settings that usually involve a trade-off between the smoothness of the domain $\Omega$ and the smoothness of the vector field $F$. For the early history, we refer to the detailed discussion in \cite{Kri54}. More advanced formulations of the divergence theorem make use of the discoveries in geometric measure theory by Caccioppoli, De Giorgi and Federer. The expositions in \cite{Mag12, EvGa15} provide a very good introduction to this topic.

In this note, we discuss a nonlocal version of divergence and normal derivative and provide a nonlocal divergence theorem analogous to \eqref{eq:div-theo}. Whereas the proof of the classical divergence theorem is quite involved, the proof of the nonlocal divergence theorem is a very simple application of Fubini's theorem. Nevertheless, by choosing a specific sequence of kernels, the nonlocal divergence and normal derivative converge to their local counterparts, see \autoref{prop:div_conv} and \autoref{thm:normal_conv}. This approach allows for a novel and elementary proof of the classical divergence theorem. 
We apply our method to give a proof in bounded $C^1$-domains and discuss how to extend it to more general domains such as polytopes. 

In order to formulate the nonlocal divergence theorem, we need to define nonlocal operators corresponding to the divergence and the inner product $F(x) \cdot \nv(x)$. To this end, we consider an even function $\alpha: \Rd \setminus \{0\} \to [0, \infty)$ and a symmetric measure $\m(h)\intd h$ with

\begin{align}\label{eq:assum_alpha-mu}
\begin{split}
\int_\Rd \min\{1, |h|^2 \} \; \alpha (h) \, \m (h) \intd h < \infty \,.
\end{split} 
\end{align}

Furthermore, we assume that $\m(h)\intd h$ is absolutely continuous with respect to $\alpha(h) \intd h$. We call a tuple $(\alpha, \m)$ with the above properties \textit{admissible}. Condition \eqref{eq:assum_alpha-mu} says that $\alpha \, \m$ is the density of a L\'{e}vy measure. 

The nonlocal divergence operator and the nonlocal normal operator for an antisymmetric function $f: \Rtd \to \R$ are defined as follows: 
\begin{alignat}{1}
    &\mathcal D_{\m} f(x) 
        \defeq  2 \, \operatorname{pv.} \int_\Rd f(x, y) \; \m (y-x) \intd y, \quad
        x \in \Omega \label{eq:def-Dmu}\\
    &\mathcal N_{\m} f(x) 
        \defeq - 2 \int_\Omega f(x, y) \; \m(y-x) \intd y,
        \quad x \in \Omega^c \label{eq:def-Nmu}
\end{alignat}

In this context, the antisymmetric function $f: \Rtd \to \R$ may be interpreted as a nonlocal stand-in for a vector field, see \autoref{lem:well-posedness-one} and  \autoref{lem:well-posedness-two} for more details. A special case of a nonlocal vector field is the nonlocal gradient vector field, defined for a scalar function $\varphi: \Rd \to \R$ by  
\begin{align}\label{eq:def-grad-nonloc}
	\mathcal G_\alpha \varphi (x, y) \defeq 
	\alpha (y-x) \, \big(\varphi (y) - \varphi (x)\big) \,.
\end{align}
In \autoref{lem:duality}, we prove duality of the operators $\mathcal{D}_{\m}$ and $\mathcal{G}_\alpha$. When writing $\mathcal D_{\m} f(x)$ or $\mathcal N_{\m} f(x)$, we always assume $f$ to be regular enough for the integrals to converge. The nonlocal divergence theorem reads: 

\begin{prop}[nonlocal divergence theorem]\label{prop:nonlocal_div}
    Let $\Omega \subset \Rd$ be measurable, $f: \Rtd \to \R$ antisymmetric and sufficiently regular and $\m(h)\intd h$ a symmetric measure. Then 
        \begin{align}\label{eq:div-nonloc}\tag{NT}
        \int_\Omega \mathcal D_{\m} f(x) \intd x = \int_{\Omega^c} \mathcal N_{\m} f(x) \intd x.
        \end{align}
\end{prop}

\begin{proof} For a symmetric function $K: \Rtd \to [0, \infty]$ and an antisymmetric function $f: \Rtd \to \R$
\begin{align*}
     \int_{\Omega} & \Big( 2 \int_{\R^d} f(x, y) 
    K(x,y) \intd y \, \Big) \intd x =  2 \int_\Omega \int_{\Omega^c} f(x, y) K(x,y) \intd y \, \intd x\\
    &= - 2 \int_\Omega \int_{\Omega^c} f(y, x) 
    K(x,y) \intd y \, \intd x =  \int_{\Omega^c}  \Big( -2 \int_\Omega f(x, y)  K(x,y)  \intd y \,\Big) \intd x \,.
\end{align*}
Equation \eqref{eq:div-nonloc} follows if we choose $K(x,y)=\mu(y-x)$. An approximation argument might be needed to show that all integrals converge. 
\end{proof}

Let us explain how we will derive the classical divergence theorem \eqref{eq:div-theo} from its simple nonlocal counterpart \eqref{eq:div-nonloc}. The main ingredient in the proof is a suitably normalized sequence of admissible pairs $(\alpha_\eps, \m_\eps)_{\eps\in(0,1)}$. Given a vector field $F \in C^1(\overline\Omega; \Rd)$, we will define  a sequence of antisymmetric functions $f_{\alpha_\eps}: \Rtd \to \R$. The main step of the proof is to show 
\begin{align}
    &\lim_{\eps\to 0} \int_\Omega \mathcal D_{\meps} f_{\aeps} (x) \intd x
    = \int_{\Omega}\diver F(x) \intd x, \label{eq:div_conv}\tag{DC} \\
    &\lim_{\eps \to 0} \int_{\Omega^c} \mathcal N_{\meps} f_{\aeps}(x) \intd x 
    = \int_{\partial\Omega} F(x) \cdot \nv(x) \intd \sigma_{d-1}(x) \,. \label{eq:normal_conv}\tag{NC}
\end{align}
Then \eqref{eq:div_conv} and \eqref{eq:normal_conv} together with \autoref{prop:nonlocal_div} imply the classical divergence theorem \eqref{eq:div-theo}.

The convergence result \eqref{eq:div_conv} is proved in \autoref{prop:div_conv} for bounded domains. 
The claim \eqref{eq:normal_conv} is more involved because it includes the nonlocal normal derivative and a change of dimension in the integration domain. Its proof for bounded $C^1$-domains is given in \autoref{thm:normal_conv} where we will choose $\aeps$ and $\m_\eps$ for $\eps \in (0,1)$ as 
\begin{align}\label{eq:case-eps}
    \aeps(h) \defeq \eps^{-1} \;\mathbbm{1}_{B_\eps(0)}(h)
    \textrm{ and } 
    \m_\eps(h) \defeq \frac{d(d+2)}{\mathcal H^{d-1}(\mathbb S^{d-1})} \eps^{-d-1} 
    \;\mathbbm{1}_{B_\eps(0)}(h) \, ,
\end{align}

where $\mathcal H^{d-1}$ is the $(d-1)$-dimensional Hausdorff measure. Note that $\aeps \m_\eps$ behaves like the localizing kernels that often crop up in the context of peridynamics (see e.g. \cite{Du13}).  Having proved \eqref{eq:normal_conv} for localizing kernels, the classical divergence theorem follows from the nonlocal theorem with (\ref{eq:div_conv}).
A version of (\ref{eq:normal_conv}) for more general functions $\aeps$ and $\m_\eps$ can be obtained from these results, see \autoref{cor:normal_conv_general}.
At the end of \autoref{sec:proof-classical-theorem}, we explain how to prove \eqref{eq:normal_conv} for more general domains such as polytopes.

Having a notion of nonlocal divergence at hand, one may employ it to reformulate interesting nonlocal objects. In \autoref{sec:frac-operators} we provide three examples for the special case

\begin{align}\label{eq:case-frac}
\alpha_\eps (h) \defeq |h|^{-1+\eps} \text{ and } \m_\eps (h) \defeq \frac {2d\eps(1-\eps)}{\mathcal H^{d-1}(\mathbb S^{d-1})} |h|^{-d-(1-\eps)} \,.
\end{align}

Both examples, \eqref{eq:case-eps} and \eqref{eq:case-frac}, share a certain normalizing property. Yet, the two examples are very different. In \eqref{eq:case-eps} $\m_\eps$ is a bounded function with compact support, which shrinks as $\eps \to 0$. In \eqref{eq:case-frac} $\m_\eps$ is singular at $h=0$ and supported in all of $\R^d$. As will become clear by the examples provided in \autoref{sec:frac-operators} it makes sense to call \eqref{eq:case-frac} the fractional case where $s=1-\eps \in (0,1)$ is the order of differentiability. Writing
\begin{align}\label{eq:def-frac-div-grad}
\operatorname{div}^{(s)} f = \mathcal{D}_{\m_{1-s}} f, \quad \nabla^{(s)} \varphi = \mathcal{G}_{\alpha_{1-s}} \varphi  
\end{align} 
with $\m$ and $\alpha$ as in \eqref{eq:case-frac}, we obtain a definition of the fractional divergence and the fractional gradient. Note that $\nabla^{(s)} \varphi (x,y) = \frac{\varphi(y) - \varphi(x)}{|y-x|^{s}}$. This notion of fractional divergence enables us to write nonlocal objects in a way that bears a structural similarity to their classical definitions. In \autoref{sec:frac-operators}, we exemplify this approach through the fractional $p$-Laplacian, the fractional perimeter, and the fractional mean curvature 

\medskip

We now discuss some related literature. 

The operators $\mathcal D_\m$ and $\mathcal N_\m$ have been introduced in several works as nonlocal counterparts of the divergence resp. the normal derivative in the context of a nonlocal vector calculus, e.g. in  \cite{GL10, Du13, ALG15, MiHi2015, DRV17, ClWa20, DEl21, CoSt22}. Different instances of these operators are obtained by choosing different kernels, i.e. by varying the functions $\alpha$ and measures $\m$ in our definition of $\mathcal D_\m$ and $\mathcal N_\m$.  See \cite{MiHi2015} for an abstract nonlocal divergence operator, where the fractional divergence appears as an example. Fractional kernels like \eqref{eq:case-frac} are generally often used in connection with problems related to real harmonic analysis or integro-differential operators. The articles \cite{MaSch18, AJS22} develop a fractional vector calculus including the fractional divergence operator $\operatorname{div}^{(s)}$. A localizing kernel like \eqref{eq:case-eps} is prominent in works related to peridynamics \cite{GL10, Du13} and appears in the approximation of minimal surfaces, perimeter, and curvature \cite{MRT19}. 

Nonlocal operators in bounded domains that do not take into account values outside of the domain may allow for integration-by-parts formulas as well. Integration by parts formulas in this case are provided and applied to linear operators in \cite{Gua06}, \cite{GaWa17} and to nonlinear operators in \cite{War16}.

As explained above, nonlocal versions of the divergence theorem can be easily shown for the operators $\mathcal D_\m$ and $\mathcal N_\m$. In \cite{Du13}[Section 5] as well as \cite{GL10}[Section 2], the connection between the local and nonlocal version of the divergence theorem is discussed. To this end, vector fields $q: \Rd \to \Rd$ are considered that can be written as an integral over a specially-defined antisymmetric function $p: \Rtd \to \R$. 
Then, it is possible to express $\int_{\partial\Omega} q(x) \cdot \nv(x) \intd \sigma_{d-1}(x)$ as $\int_{\Omega^c} \int_\Omega p(x, y) \intd y \intd x $. 
Since this identity is proved with the help of the the local divergence theorem \eqref{eq:div-theo}, this approach does not provide a proof of \eqref{eq:div-theo}.

The nonlocal normal derivative is often brought up in the context of Green's first identity, also called the Gauss-Green formula. Let $\nu : \R^d \setminus \{0\} \to [0, \infty)$ be a symmetric function. For functions $\varphi, \psi: \Rd \to \R$, we define 
\begin{align*}
L_\nu  \varphi(x) &\defeq 2 \textrm{ pv. }  \int_{\R^d} \big(\varphi(x) - \varphi(y)\big) \nu (x-y) \intd y, \quad x \in \Omega\,, \\
N_\nu  \varphi (y) &\defeq 2 \int_\Omega \big(\varphi(y) - \varphi(x)\big) \nu (x-y) \intd x, \quad y \in \Omega^c \,, \\
E_\nu (\varphi, \psi) &\defeq \iint_{(\Omega^c \times\Omega^c)^c} \big(\varphi(x)-\varphi(y)\big)\big(\psi(x)-\psi(y)\big) \nu (x-y) \intd x \intd y \,.
\end{align*}
Then the nonlocal Gauss-Green formula reads
\begin{align}  
\int_\Omega L_\nu \varphi(x) \psi(x) \intd x = E_\nu  (\varphi, \psi) - \int_{\Omega^c} N_\nu  \varphi(y) \psi(y) \intd y.
\end{align}
If we set $f(x, y) = \varphi(x) - \varphi (y)$, then we obtain $L_\nu  \varphi(x) = \mathcal D_{\nu} f(x)$ and $N_\nu  \varphi(y) = \mathcal N_{\nu} f(y)$. A very general formulation of the nonlocal Gauss-Green theorem is presented in \cite{FK22}. 
Using appropriate sequences of kernels $\nu_\eps$ with properties similar to \eqref{eq:seq-levy1} and \eqref{eq:seq-levy2}, the convergence of the operators $L_{\nu_\eps}$, $E_{\nu_\eps}$, and $N_{\nu_\eps}$ to their respective local counterparts is discussed. 
More explicitly, it is shown that
$\int_\Omega L_{\nu_\eps} \varphi(x) \psi(x) \intd x$ converges to $\int_\Omega - \Delta \varphi(x) \psi(x) \intd x$ and
$E_{\nu_\eps} (\varphi, \psi)$ converges to $ 2 \int_\Omega \nabla \varphi (x) \cdot \nabla \psi (x) \intd x$
as $\eps \to 0^+$. 
The local Gauss-Green formula then yields the convergence of the nonlocal normal derivative to the local normal derivative. 
A similar but more specific approach that only considers fractional kernels is pursued in \cite{DRV17}[Section 5.1].
Note that both approaches use the local version of the Gauss-Green theorem to establish the convergence of the nonlocal to the local normal derivative. An interesting Gauss-Green formula for nonlocal operators may be found in \cite{Scott23}[Theorem 1.1]. Although the operator $L$ and the bilinear form $\mathcal{E}$ are nonlocal, the classical normal derivative on the boundary of the domain is retained, which is due to the specific choice of the kernel.

\medskip

This note is structured as follows. In \autoref{sec:setting} we develop the nonlocal setup in more detail and discuss some basic properties of the nonlocal operators. The first part of \autoref{sec:proof-classical-theorem} contains our first main result \eqref{eq:div_conv}. The classical divergence theorem \eqref{eq:div-theo} then follows from \eqref{eq:normal_conv} resp. \autoref{thm:normal_conv}, which is proved in the second part of \autoref{sec:proof-classical-theorem}. In \autoref{sec:frac-operators} we apply our notion of fractional divergence and re-write well studied quantities like the fractional $p$-Laplace, the fractional perimeter, and the fractional mean curvature.

\medskip

\textbf{Acknowledgments:} The authors thank Florian Grube for proofreading and Michael Hinz for fruitful discussions.

\section{The nonlocal setting}\label{sec:setting}

In this section, we discuss the nonlocal operators $\mathcal{D}_{\m}, \mathcal{N}_{\m}, \mathcal{G}_\alpha$ and state some fundamental properties. We call an open bounded connected subset of $\R^d$ a domain and always assume $\Omega$ to be a domain in this sense.  Given $\alpha$ and $\m$ satisfying \eqref{eq:assum_alpha-mu}, we can explicitly state regularity conditions for an antisymmetric function $f$ that guarantee the existence of the integral $\mathcal D_{\m} f(x)$. Note that $\mathcal{N}_{\m} f(x)$ is defined without any smoothness assumption on $f$.

\begin{lem}\label{lem:well-posedness-one} Let $\m (h) \intd h$ be a symmetric measure, $f: \Rtd \to \R$ an antisymmetric measurable bounded function, $x \in \R^d$, and $M_1, M_2 > 0$. Assume
	\begin{align}\label{eq:well-posedness-div-ness1}
	&\int_{B_1} \big| f(x, x+h) + f(x,x-h) \big| \m (h) \intd h \leq M_1 \, ,\\
	\label{eq:well-posedness-div-ness2}
	&\int_{B_1^c} \big| f(x, x+ h) \big| \, \m(h) \intd h \leq M_2 \,.
	\end{align}
Then $\mathcal{D}_{\m} f (x)$ exists. 
\end{lem}

\begin{proof} Let $\eps \in (0, 1)$. For $x \in \R^d$, there holds
\begin{align*}
	&\int_{B^c_\eps (x)} f(x, y) \m (y-x) \intd y = \int_{B_\eps^c} f(x, x+h) \m (h) \intd h \\
	&= \int_{B_\eps^c \cap B_1} f(x, x+h) \m(h) \intd h + \int_{B_1^c} f(x, x+h) \m(h) \intd h\\
	&= \frac12 \int_{B_\eps^c \cap B_1} f(x, x+h) - 2 f(x,x) + f(x,x-h) \m (h) \intd h + \int_{B_1^c} f(x, x+h) \m(h) \intd h \\
	&\leq \frac 12 \int_{B_1} \big| f(x, x+h) + f(x,x-h) \big| \m (h) \intd h + M_2
	\leq M_1 + M_2 \,.
\end{align*}	
This proves the existence of the principal value integral in the definition of $\mathcal{D}_{\m} f (x)$.
\end{proof}

We remark that condition \eqref{eq:well-posedness-div-ness1} couples the integrability of $\m (h)$ at $h=0$ with the regularity of the function $f(x,y)$ at the diagonal $x=y$. Condition \eqref{eq:well-posedness-div-ness2} ensures the integrability of $f(x, x+ h)$ with respect to $\m(\intd h)$ for $h$ away from zero. 

The duality of divergence and gradient is naturally recovered in our nonlocal setup. Recall the definition of the nonlocal divergence $\mathcal{D}_{\m}$ in \eqref{eq:def-Dmu}. We recall the standard scalar product for functions $\varphi, \psi: \Rd \to \R$ as 
$\langle \varphi, \psi \rangle \defeq \int_\Rd \varphi (x) \psi(x) \intd x$, and define a nonlocal scalar product for functions $f, g: \Rtd \to \R$ of two arguments as 
$\langle f, g \rangle_\m \defeq \int_\Rd \int_\Rd f(x, y) g(x, y) \m(y-x) \intd y \intd x.$
This immediately yields the following lemma.
\begin{lem}\label{lem:duality}
	Let $\alpha: \Rd \to \R$ be even, $\m(h)\intd h$ a symmetric measure, $f: \Rtd \to \R$ antisymmetric and sufficiently regular, and $\varphi: \Rd \to \R$ a scalar function. 
	Then 
	$$\langle \mathcal D_{\m} f, \varphi \rangle = - \langle f, \mathcal G_\alpha \varphi \rangle_{\alpha^{-1} \m} \,.$$
\end{lem}

\begin{proof} The proof is a simple calculation:
\begin{alignat*}{1}
	\langle f, \mathcal G_\alpha \varphi \rangle_{\alpha^{-1} \m} 
	=& \int_\Rd \int_\Rd f(x,y) \big( \varphi(y) - \varphi(x) \big) \alpha (y-x) \alpha^{-1}(y-x) \m (y-x) \intd y \intd x\\
	=&- \int_\Rd \int_\Rd f(y, x) \varphi(y) 
	\m (y-x) \intd y \intd x\\
	&- \int_\Rd \int_\Rd f(x, y) \varphi (x) \m (y-x) \intd y \intd x \\
	=& - 2 \int_\Rd \varphi (x) \int_\Rd f(x, y)
	\m (y-x) \intd y \intd x
	= - \langle \mathcal D_{\m} f, \varphi \rangle.
\end{alignat*}
\end{proof}

In order to connect the local with the nonlocal divergence theorem, we need to construct an antisymmetric function $f: \Rtd \to \R$ from a given vector field $F: \overline{\Omega} \to \Rd$ that can be used in the operators $\mathcal D_{\m}$ and $\mathcal N_{\m}$. We will permit ourselves a little terminological leeway and call 
\begin{align}\label{eq:nonloc_vf}\tag{VF}
    f_\alpha (x, y) \defeq \alpha (x - y) \int_0^1 F\big(x + t (y-x)\big) \cdot (y-x) \intd t\,, \quad x,y \in \R^d, x\ne y
\end{align}
the \textit{nonlocal vector field generated by} $F$. We always set $f_\alpha (x, x) = 0$. In this definition, we assume $F$ to be extended to a compactly supported vector field on $\R^d$, see also \cite[Theorem 6.2.4]{stein}. The specific choice of the extension does not impact our arguments. Observe that $f_\alpha$ is antisymmetric. Since the extension of $F$ is assumed to have compact support, we know that $f_\alpha$ will be bounded on $\Rtd$ for any $\alpha$. If $F(x)$ describes a force vector at $x$, then one might interpret the scalar $f_\alpha (x,y)$ as the aggregated directed magnitude of $F$ along the line between $x$ and $y$.
It is easy to see that the nonlocal gradient vector field $\mathcal G_\alpha \varphi$ is generated by the classical gradient vector field $\nabla \varphi: \Rd \to \Rd$ in the sense of \eqref{eq:nonloc_vf}.

We state conditions on $F$, $\alpha$, and $\m$ that guarantee the existence of the integral $\mathcal D_{\m} f_\alpha (x)$.

\begin{lem}\label{lem:well-posedness-two} Let $(\alpha, \m)$ be admissible in the sense of \eqref{eq:assum_alpha-mu}. Assume $F \in C^1(\Omega; \Rd)$ and let $f_\alpha$ be the nonlocal vector field generated by $F$ with regard to $\alpha$ as in \eqref{eq:nonloc_vf}. Then $\mathcal D_{\m} f_\alpha (x)$ exists at every point $x \in \Omega$.
\end{lem}

\begin{proof}
We verify conditions \eqref{eq:well-posedness-div-ness1} and \eqref{eq:well-posedness-div-ness2} of \autoref{lem:well-posedness-one}. We begin with \eqref{eq:well-posedness-div-ness1}. Let $x \in \Omega$. Using the fact that $F \in C^1(\Omega; \Rd)$, we have
\begin{align*}
	\int_{B_1} &\big| f_\alpha(x, x + h) + f_\alpha(x, x-h) \big| \m(h) \intd h \\
	&= \int_{B_1} \Big| \int_0^1 F(x+th) \cdot h + F(x-th)\cdot (-h)\big) \intd t \, \alpha(h) \Big| \m(h) \, \intd h\\
	&\leq || F ||_{C^1} \int_{B_1} |h|^2 \, \alpha(h) \m(h) \intd h < \infty 
\end{align*}
by \eqref{eq:assum_alpha-mu}. Next, let us verify \eqref{eq:well-posedness-div-ness2}. Since $\Omega$ is bounded and $F$ has compact support in $\R^d$, there is some constant $R$ such that $F(x + th) = 0$ for any $t \in [0, 1]$, $x \in \Omega$, and $|h| > R$. We see that 
\begin{align*}
	\int_{B_1^c} & \big|f_\alpha (x, x+ th) \big| \, \m(h) \intd h \leq 
	\int_{B_R \setminus B_1} \int_0^1 | F(x+th) \cdot h| \, \intd t \, \alpha(h) \m(h) \intd h\\
	&\leq R ||F||_{C^1} \int_{B_R\setminus B_1} \min\{ 1, |h|^2\} \alpha(h) \m(h) \intd h < \infty, 
\end{align*}
using \eqref{eq:assum_alpha-mu} once more. We have verified conditions \eqref{eq:well-posedness-div-ness1} and \eqref{eq:well-posedness-div-ness2}. \autoref{lem:well-posedness-one} completes the proof. 
\end{proof}

In order to prove the convergence of the nonlocal operators to their local counterparts, we generate whole families of nonlocal vector fields from the same vector field $F: \Omega \to \Rd$, using families of radial functions $(\alpha_\eps)_\eps$ and symmetric measures $(\m_\eps(h) \intd h)_\eps$. 
These families are closely related to the L\'{e}vy measures introduced, for example, in \cite{Fog21} or \cite{FK22}. 
Specifically, we consider radial functions $\big(\aeps: \Rd \setminus \{0\} \to \R \big)_{\eps \in (0, 1)}$ and symmetric measures $\m_\eps(h) \intd h$ (where the functions $\m_\eps$ are also radial) that satisfy the following properties: 
\begin{align}
    &\forall \eps > 0: \int_\Rd \min\{1, |h|^2 \} \; \aeps(h) \m_\eps (h) \intd h = d
    \label{eq:seq-levy1}\tag{L1} \,,\\
    &\forall \delta > 0: \lim_{\eps \to 0^+} \int_{|h| > \delta} \aeps(h) \m_\eps (h) \intd h = 0\label{eq:seq-levy2}\tag{L2}
\end{align}

The family  \eqref{eq:case-eps} of functions related to bounded localizing kernels satisfies \eqref{eq:seq-levy1} and \eqref{eq:seq-levy2}, as does the family related to the fractional case given in \eqref{eq:case-frac}.
When nonlocal vector fields generated from a local vector field $F$ as in \eqref{eq:nonloc_vf} are defined with respect to a whole family $(\aeps)_{\eps \in (0,1)}$, $(\meps)_{\eps \in (0,1)}$ instead of only single functions $\alpha$ resp. $\m$, we write $f_{\eps}$ instead of $f_{\alpha_\eps}$ and $\divop$ resp. $\Nop$ instead of $\mathcal D_{\meps}$ resp. $\mathcal N_{\meps}$.

\section{The local divergence theorem}\label{sec:proof-classical-theorem}

In this section we give the proof of the local divergence theorem using the nonlocal divergence theorem. We do this by verifying the convergences \eqref{eq:div_conv} and \eqref{eq:normal_conv}. First, we prove \eqref{eq:div_conv} for families of functions that satisfy \eqref{eq:seq-levy1} and \eqref{eq:seq-levy2}. Then we complete the proof of the local divergence theorem \eqref{eq:div-theo} by proving \eqref{eq:normal_conv}.

\subsection*{Proof of \eqref{eq:div_conv}} The following proposition proves the convergence result \eqref{eq:div_conv}.

\begin{prop}\label{prop:div_conv}
    Assume $\Omega \subset \Rd$ is open and bounded and take $(\aeps)_{\eps\in(0,1)}$ and $(\m_\eps (h) \intd h)_{(\eps \in (0, 1)}$ to be families of radial functions resp. symmetric measures satisfying (\ref{eq:seq-levy1}) and (\ref{eq:seq-levy2}). 
    Let $F \in C^1(\overline\Omega; \Rd)$ be a vector field and $(f_\eps: \Rtd \to \R)_{\eps \in (0,1)}$ be a family of nonlocal vector fields  with respect to $(\aeps)$ in the sense of (\ref{eq:nonloc_vf}). 
    Then 
    \begin{align*}
        \int_\Omega \divop f_\eps(x) \intd x \to \int_\Omega \diver F(x) \intd x
    \end{align*}
    as $\eps \to 0^+$.
\end{prop}

Our proof is inspired by the proof of Proposition 2.4 in \cite{Fog21}.

\begin{proof}
Let $x \in \Omega$ be arbitrary. 
We begin by splitting the domain of integration into a part where $x$ and $y$ are close together and a part where they are not. 
For any $\delta \in (0,1)$, write: 
\begin{alignat*}{1}
    \divop f_\eps (x)
    =& \; 2 \int_\Rd 
        \int_0^1 F(x + t(y-x)) \cdot (y-x) \intd t 
        \;\aeps(y-x) \m_\eps(y-x) \intd y\\
    =& \; 2
        \int_{|h| \leq \delta} 
        \int_0^1 F(x + th) \cdot h \intd t 
        \;\aeps(h) \m_\eps(h) \intd h\\
    &+ \; 2
        \int_{|h| > \delta} 
        \int_0^1 F(x + th) \cdot h \intd t 
        \;\aeps(h) \m_\eps(h) \intd h
\end{alignat*}

Now the second integral with $|h| > \delta$ goes to zero. 
Indeed, $\int_0^1 F\big(x + th\big) \cdot h \intd t$ is bounded since $F$ has compact support in $\R^d$.  
Hence, using (\ref{eq:seq-levy2}), we see that 
$$\int_{|h| > \delta} \int_0^1 |F(x+th)| |h| \intd t \; \aeps(h) \m_\eps(h) \intd y 
< c \int_{|h| > \delta} \; \aeps(h) \m_\eps(h) \intd y$$ 
vanishes as $\eps \to 0$.

For the part where $|h| \leq \delta$, we use the symmetry of the integral as well as the fundamental theorem of calculus to obtain
\begin{alignat*}{1}
    2& \; \frac 1 2 \int_{|h| \leq \delta}
        \int_0^1 [F(x+th) - F(x-th)] \cdot h \intd t \;
        \:\aeps(h) \m_\eps(h)  \intd h\\
    &=  \int_{|h| \leq \delta} 
        2 \int_0^1 t \int_0^1 [DF(x - th + 2sth)h] \cdot h \; \intd s \intd t
        \:\aeps(h) \m_\eps(h) \intd h\\
    &=  \int_{|h| \leq \delta} 
        2 \int_0^1 t \int_0^1 [DF(x)h] \cdot h \; \intd s \intd t
        \:\aeps(h) \m_\eps(h) \intd h\\
    &\quad +  \int_{|h| \leq \delta} 
        2 \int_0^1 t \int_0^1 [DF(x-th + 2sth) - DF(x)]h \cdot h \intd s \intd t
        \:\aeps(h) \m_\eps(h) \intd h.
\end{alignat*}
By assumption, $DF$ is bounded on $\Rd$. Thus, for any $\eta > 0$, we find a $\delta \in (0,1)$ sufficiently small, such that any $|h| \leq \delta$ gives $|DF(x + (2st-t)h) - DF(x)| < \eta$. This allows us to estimate the second term by 
\begin{alignat*}{1}
    \int_{|h| \leq \delta} \eta |h|^2 \aeps(h) \m_\eps(h) \intd h \leq  \; \eta \int_\Rd \min\{1, |h|^2\} \aeps (h) \m_\eps(h) \intd h 
    = \eta d,
\end{alignat*} 
using \eqref{eq:seq-levy1}. Since $\eta$ was arbitrary, the second term must be zero. 

Now we treat the first term. 
For any $x \in \Omega$, we have 
    \[ [DF(x)h] \cdot h = \sum_{j=1}^d \sum_{i=1}^d \frac {\partial F_j} {\partial x_i} (x) h_i h_j.\]
By assumption, $\aeps \m_\eps$ is radial, which means that $\int_{|h| \leq \delta} h_i h_j \aeps(h) \m_\eps(h) \intd h = 0$ for $i \neq j$. 
We are thus allowed to disregard all terms in the foregoing sum where $i \neq j$. This yields: 
\begin{alignat*}{1}
    2&  \int_{|h| \leq \delta} 
        \int_0^1 t \int_0^1 [DF(x)h] \cdot h \; \intd s \intd t \;
        \aeps(h) \m_\eps(h) \intd h\\
    &=  \int_{|h| \leq \delta} 
        \sum_{j=1}^d \frac {\partial F_j} {\partial x_j} (x) h_j^2 \;
        \aeps(h) \m_\eps(h) \intd h\\
    &=  \sum_{j=1}^d \frac {\partial F_j} {\partial x_j} (x) 
        \int_{|h| \leq \delta}  h_1^2 \;
        \aeps(h) \m_\eps(h) \intd h\\
    &=  \diver F(x)
        \frac 1 d \int_{|h| \leq \delta}  |h|^2 \;
        \aeps(h) \m_\eps(h) \intd h
\end{alignat*}
For any $\delta \in (0,1)$, there holds
\begin{alignat*}{1}
    &\lim_{\eps \to 0} \int_{|h| \leq \delta} \min\{1, |h|^2\} \aeps(h) \m(h)\intd h 
    = d - \lim_{\eps \to 0} \int_{|h| > \delta} \min \{ 1, |h|^2 \} \aeps(h) \m(h)\intd h = d,
\end{alignat*}
using (\ref{eq:seq-levy1}), (\ref{eq:seq-levy2}), and the fact that $\min\{1, |h|^2\} \leq 1$. Thus
\[2  \int_{|h| \leq \delta} 
        \int_0^1 t \int_0^1 [DF(x)h] \cdot h \intd s \intd t \;
        \aeps(h) \m_\eps(h) \intd h \to \diver F(x) \text{ as } \eps \to 0.\]
Putting everything together produces the claim. 
\end{proof}

An interesting consequence of \autoref{prop:div_conv} is the convergence of the nonlocal scalar products introduced in \autoref{sec:setting} to their respective local counterparts. 
Precisely, we get for $\varphi \in C^2(\Rd)$,
$\langle \divop f_\eps, \varphi\rangle \to \langle \diver F, \varphi \rangle$ as $\eps \to 0$ 
by \autoref{prop:div_conv}. 
Local and nonlocal duality of divergence and gradient then provide the limit of the other scalar product,
$\langle f_\eps, \mathcal G_\eps \varphi \rangle_{\alpha_\eps^{-1}\m_\eps} \to \langle F, \nabla \varphi \rangle$ as $\eps \to 0$.

\subsection*{Proof of \eqref{eq:normal_conv}}
We will now complete the proof of the local divergence theorem \eqref{eq:div-theo} making use of the nonlocal divergence theorem \eqref{eq:div-nonloc}. Since \eqref{eq:div_conv} has already been established in \autoref{prop:div_conv}, it remains to prove \eqref{eq:normal_conv}. Note that \autoref{prop:div_conv} was stated for a general family of radial functions $(\aeps)$ and symmetric measures $(\m_\eps(h) \intd h)$ that satisfy \eqref{eq:seq-levy1} and \eqref{eq:seq-levy2}.  For \autoref{thm:normal_conv} we do not need this generality. 
The classical divergence theorem will follow from its nonlocal counterpart already if we can prove \eqref{eq:normal_conv} for specific families $(\aeps)$ and $(\m_\eps)$. 
We choose $(\aeps)$ and $(\m_\eps)$ as defined in (\ref{eq:case-eps}) and write for brevity $a_d \defeq d \frac{d+2}{\mathcal H^{d-1}(\mathbb S^{d-1})}$. 

The proof of (\ref{eq:normal_conv}) is provided in \autoref{thm:normal_conv} and has three major steps. 
First, we localize the problem using the fact that $\Omega$ is a $C^1$-domain. 
Next, we argue that for a point $x \in \Omega^c$, $\Nop f(x)$ evaluates to the inner product of $F(x) \cdot \nv(z_x)$, where $\nv(z_x)$ is the unit outward normal vector to $\partial\Omega$ at a point $z \in \partial\Omega$ that minimizes $\dist(x, \Omega)$. 
In the third step, we let $x \in \Omega^c$ approach the boundary to obtain the classical surface integral. 
For this step, we need an approximate identity defined in Lemma \ref{lemma:approx_id} that will accomplish the dimension collapse of $\Omega^c$ to $\partial\Omega$ in the limit. 
We are aware of the fact that similar constructions have been used before, see e.g. \cite[Theorem 6.1.5]{Hor90} or \cite[Theorem 3.2.39]{Fed69}.

For $t\in \R$, we introduce the shorthand notation $B_\eps^{t} \defeq \{h \in B_\eps(0) \mid h_d < t \}$. Note that if $t = 0$, $B_\eps^{t}$ is just the lower half-ball.

\begin{lem}\label{lemma:approx_id}
    The family of functions 
    \[\Big(k_\eps(t) \defeq - 2d \frac {d+2}{\mathcal H^{d-1}(\mathbb S^{d-1})} \; \eps^{-d-2} \int_{B_\eps^{-t}} h_d \, \intd h \; \mathbbm{1}_{(0, \eps)}(t)\Big)_{\eps\in(0,1)}
    \]
    is an approximate identity in 0.
\end{lem}

\textbf{Proof}.
It is immediately clear that for any $\delta > 0$, $\int_{|t| > \delta} k_\eps (t) \intd t = 0$ once $\eps < \delta$, hence $\lim_{\eps\to0^+} \int_{|t| > \delta} k_\eps(t) \intd t = 0.$ Since $k_\eps \geq 0$ for any $\eps \in (0, 1)$, it remains to be shown that $\int_\R k_\eps(t) \intd t = 1$ for any $\eps \in (0,1)$.

Observe that the integral over $B^{-t}_\eps$ can be rewritten as
\begin{alignat*}{1}
	\int_{B_\eps^{-t}} h_d \, \intd h 
	    &= \int_{-\eps}^{-t} r |B_1^{d-1}| (\eps^2 - r^2)^{\frac {d-1} 2} \intd r \\	    
	    &= - \frac 1 2 |B_1^{d-1}| \int_0^{\eps^2 - t^2} w^{\frac {d-1} 2} \intd w 
	    = - \frac 1 {(d+1)} |B_1^{d-1}| (\eps^2 - t^2)^{\frac{d+1} 2},
\end{alignat*}
substituting $w\defeq \eps^2 - r^2$. 
Here $B_1^{d-1}$ is just the $(d-1)$-dimensional unit ball. 
There holds $|B_1^{d-1}| = \pi^{\frac{d-1}{2}} / \Gamma(\frac{d+1}{2})$ as well as  $\mathcal H^{d-1}(\mathbb{S}^{d-1}) = d\pi^{\frac{d}{2}}/\Gamma(\frac{d+2}{2})$, where $\Gamma$ is Euler's gamma function. Hence
\begin{alignat*}{1}
	\int_\R k_\eps(t) \intd t 
	&= \frac{2d(d+2)}{(d+1)} \; \eps^{-d-2} \; 
        \frac {|B_1^{d-1}|}{\mathcal H^{d-1}(\mathbb S^{d-1})} 
		\int_0^\eps (\eps^2 - t^2)^{\frac{d+1} 2} \intd t \\
	&= \frac{2d(d+2)}{(d+1)} \; \eps^{-d-2} \; 
        \frac {|B_1^{d-1}|}{\mathcal H^{d-1}(\mathbb S^{d-1})} 
		\frac{\pi^{\frac 1 2}\eps^{d+2} \Gamma(\frac{d+3}{2})}{2\Gamma(\frac{d+4}{2})}\\
	&= \frac{(d+2)}{(d+1)} \; 
		\frac{\Gamma(\frac{d+2}{2})\Gamma(\frac{d+3}{2})}{\Gamma(\frac{d+1}{2})\Gamma(\frac{d+4}{2})} = 1.
\end{alignat*}
In the last line, the basic fundamental equality $\Gamma(x+1) = x\Gamma(x)$ for Euler's gamma function was used. \niceQed

\begin{thm}\label{thm:normal_conv}
    Assume $\Omega \subset \Rd$ to be a bounded $C^1$-domain and let $(\aeps)_{\eps\in(0,1)}$ as well as $(\m_\eps)_{\eps \in (0,1)}$ be defined as in (\ref{eq:case-eps}). 
    Let $F \in C^1(\overline\Omega; \Rd)$ and $(f_\eps: \Rtd \to \R)_{\eps \in (0,1)}$ be a corresponding family of nonlocal vector fields with respect to $(\aeps)_\eps$ in the sense of (\ref{eq:nonloc_vf}). Then 
    \begin{align*}
        \int_{\Omega^c} \Nop f_\eps(x) \intd x \to \int_{\partial\Omega} F(x) \cdot \nv(x) \intd\sigma_{d-1}(x)
    \end{align*}
    as $\eps \to 0^+$.
\end{thm}

An interesting consequence is the following. By combining the local and nonlocal divergence theorems with \autoref{prop:div_conv}, we can extend \autoref{thm:normal_conv} to general families $(\aeps)_\eps$ and $(\m_\eps(h) \intd h)_\eps$ that satisfy conditions (\ref{eq:seq-levy1}) and (\ref{eq:seq-levy2}). More specifically, there holds:

\begin{cor}\label{cor:normal_conv_general}
	Assume $\Omega$ to be a bounded $C^1$-domain. Let $(\aeps)_{\eps\in(0,1)}$ and $(\m_\eps (h) \intd h)_{\eps \in (0, 1)}$ be families of radial functions resp. symmetric measures satisfying (\ref{eq:seq-levy1}) and (\ref{eq:seq-levy2}). 
	Let $F \in C^1(\overline\Omega; \Rd)$ and $(f_\eps: \Rtd \to \R)_{\eps \in (0,1)}$ be a corresponding family of nonlocal vector fields with respect to $(\aeps)_\eps$ in the sense of (\ref{eq:nonloc_vf}).
	Then
	\begin{align*}
		\int_\Omega \Nop f_\eps (x) \intd x \to \int_{\partial\Omega} F(x) \cdot \nv(x) \intd \sigma_{d-1} (x)
	\end{align*}
	as $\eps \to 0$.
\end{cor}

We provide the proof for bounded $C^1$-domains. At the end of the section, we explain how our method extends to more general domains with a boundary that consists of finitely many $C^1$-parts. 

\begin{proof}[Proof of \autoref{thm:normal_conv}]
We begin by localizing. As usual, write $x = (x', x_d) \in \R^{d-1}\times\R$. 
By assumption, $\Omega$ is a $C^1$-domain, so for any $r > 0$ and $z_0 \in \partial\Omega$, we find a cube $\hat Q_r$ centered at $z_0$ of radius $r$, a rotation and translation $\phi: \R^d \to \R^d$, and a $C^1$-function $\gamma: \R^{d-1} \to \R$ such that, if $Q_r = \phi(\hat Q_r)$, 
\begin{alignat*}{1}
	&Q_r^- \defeq \{ x \in Q_r \mid \gamma(x') > x_d \} = \phi(\hat Q_r \cap \Omega),\\
	&Q_r^+ \defeq \{ x \in Q_r \mid \gamma(x') \leq x_d \} = \phi(\hat Q_r \cap \Omega^c),
\end{alignat*}
as well as $\Gamma_\gamma \cap Q_r = \phi(\hat Q_r \cap \partial\Omega)$, where $\Gamma_\gamma$ signifies the graph of $\gamma$. 

We now choose a finite subcover  $\hat Q_\rho^{(1)}, \dots, \hat Q_\rho^{(N)}$ of $\partial \Omega$ with the following properties. 
\begin{enumerate}
	\item \label{def:cover-rotation}
		The outward unit normal vector in each cube does not oscillate too much. Specifically, we require for any $z, w \in \hat Q_r \cap \partial\Omega$ that $\arccos(z \cdot w) \in [0, \pi/4)$. 
	\item \label{def:cover-minimizer}
		Each point in a smaller cube $\hat Q_\rho$ with $\rho<r$ has a minimizer in $\hat Q_r \cap \partial \Omega$. 
\end{enumerate} 
The outward unit normal vector field $\nv: \partial \Omega \to \mathbb S^{d-1}$ of a $C^1$-domain is continuous, and, since $\partial \Omega$ is compact, even uniformly. We thus choose $r > 0$ such that property \eqref{def:cover-rotation} holds.
Property \eqref{def:cover-minimizer} is obtained by simply choosing $\rho < r/2$.  
Now, $\partial\Omega$ can be covered by finitely many cubes $\hat Q_\rho^{(1)}, \dots, \hat Q_\rho^{(N)}$, again by compactness. By choosing $\eps_0 < r/2$ small enough, we can even guarantee for any $\eps \in (0, \eps_0$) that $\bigcup_{i=1}^N \hat Q_\rho^{(i)} \supset \Omega^c_\eps \defeq \{ x \in \Omega^c \mid \dist(x, \partial\Omega) < \eps \}$, as well as $B_\eps(x) \subset \hat Q^{(i)}_r$ for any $x \in \hat Q_\rho^{(i)}$.

From now on, let $\eps \in (0, \eps_0)$. 
It suffices to consider a localized version of the problem, i.e. we show for any $i \in \{1, \dots, n\}$
\begin{align}\label{eq:NC-loc}
	\lim_{\eps \to 0} \int_{\hat Q^{(i)}_\rho \cap \Omega^c_\eps} \mathcal N_{\eps} f_{\eps}(x) \intd x 
    = \int_{\hat Q^{(i)}_\rho \cap \partial\Omega} F(x) \cdot \nv(x) \intd \sigma_{d-1}(x) \,.\tag{NC-loc}
\end{align}
We then choose a partition of unity $(\eta_i)_{i= 1, \dots, N}$ of $\Omega_\eps^c$ subordinate to the covering $\hat Q^{(1)}_\rho, \dots, \hat Q^{(N)}_\rho$. Carefully following the proof of \eqref{eq:NC-loc} reveals that it implies 
\begin{align*}
	\lim_{\eps \to 0} \int_{\Omega_\eps^c} \mathcal N_{\eps} f_{\eps}(x) \eta_i(x) \intd x 
    = \int_{\partial\Omega} \big(F(x) \cdot \nv(x)\big) \eta_i(x) \intd \sigma_{d-1}(x) \,,
\end{align*}
for any $i \in \{1, \dots, N\}$. Thus, by using the definition of the kernel \eqref{eq:case-eps} and the fact that $\sum_{i=1}^N \eta_i(x) = 1$ pointwise for any $x \in \Omega_\eps^c$, we obtain \eqref{eq:normal_conv}.

In order to prove \eqref{eq:NC-loc}, let us consider a single cube $\hat Q$ in the cover of $\Omega_\eps^c$. We write $Q \defeq (-\rho, \rho)^d$ for the cube $\phi(\hat Q)$ and assume, for ease of notation and without loss of generality, that $\Omega$ is shifted and rotated in such a way that $Q^- = \Omega \cap Q$, $Q^+ = \Omega^c \cap Q$, and $\partial\Omega = \Gamma_\gamma \cap Q$. Setting $Q_\eps^+ \defeq \{ x \in Q^+ \mid \dist(x, \partial\Omega) < \eps \}$, we may also assume that for any $x' \in [-\rho, \rho)^{d-1}$, we have $\{x'\} \times (\gamma(x'), \gamma(x') + \eps) \subset Q_\eps^+.$

For $x \in Q_\eps^+$, let $z_x \in \partial\Omega$ be a point minimizing the distance of $x$ to $\partial\Omega$.
Note that there might not be a unique minimizer for every $x \in Q^+_\eps$, no matter how small we choose $\eps$ (see \cite[Section 4]{KP81}). However, this won't pose a problem, since every sequence of minimizers for $x$ converges to the same point as $x$ approaches the boundary.  
Consider the normed vector $\nv(z_x) \defeq \frac{x-z_x}{|x-z_x|}$. 
By construction, this is the outward unit normal vector to $\partial\Omega$ at $z_x$, which exists everywhere since $\partial\Omega$ is $C^1$. 
There is a rotation $R_x = (r^{(x)}_{i, j})_{i, j = 1, \dots d} \in \mathcal O^+(\R^d)$ rotating $\nv(z_x)$ to $e_d = (0, \dots, 0, 1)^T$, i.e. $R_x\nv (z_x) = e_d$.
Observe that, since $R_x$ is orthogonal, we have that $R_x^{-1}e_d = R_x^Te_d = (r^{(x)}_{d, 1}, \dots, r^{(x)}_{d, d})^T = \nv(z_x)$. 
Moreover, since the normal vector does not oscillate too much in $Q$, $R_x(\Gamma_\gamma \cap Q)$ is still the graph of a $C^1$-function.

With this setup, we begin to calculate: 
\begin{align*}
    &\int_{Q^+} \mathcal N_{\eps} f_{\eps}(x) \intd x\\
    =& -2 \int_{Q^+} \int_\Omega f_{\aeps}(x, y) \m_\eps(y-x) \intd y \intd x\\
	=&- 2 \int_{Q^+} \int_\Omega 
	    \int_0^1 F\big(x + t(y-x)\big) \cdot (y - x) \intd t \;
        \aeps(y-x) \m_\eps(y-x) \intd y \intd x\\
	=&-2 a_d \eps^{-d-2}
	    \int_{Q^+_\eps} \int_{(\Omega-x) \cap B_\eps(0)}
	    \int_0^1 F(x + th) \cdot h \, \intd t \intd h \intd x\\
	=& -2 a_d \eps^{-d-2}
	    \int_{Q^+_\eps} \int_{R_x(\Omega-x) \cap B_\eps(0)}
	    \int_0^1 F(x + t R_x^T h) \cdot R^T_x h \, \intd t \intd h \intd x\\
	=& -2 a_d \eps^{-d-2} 
		\int_{Q^+_\eps} \int_{B_\eps^{-|x-z_x|}} 
		\int_0^1 F\big(x + t R^T_x \psi_x^{-1}(h)\big) \cdot R^T_x \psi_x^{-1}(h) \intd t \intd h \intd x
\end{align*}

In the fourth equality, we have used a change of variables with $h \defeq y-x$ as well as the fact that, by definition, $\aeps$ and $\meps$ have support in $B_\eps(0)$. Note that the change of variables shifts $z_x$ to $z_x-x$ and $R_x(z_x-x) = -R_x^T\nv(z_x) |x-z_x|= -e_d |x-z_x|$. 
 
In the fifth equality, we have straightened the boundary in the following way. 
Abusing notation slightly, let us also call $\gamma$ the $C^1$-function on $R_x(\Omega-x)\cap B_\eps(0)$ that represents $\partial \Omega$. Define a $C^1$-function 
\begin{align*}
	\psi_x: V_x \to \Rd, (h', h_d) \mapsto (h', h_d - \gamma(h') - |x - z_x|),
\end{align*}
in a neighbourhood $V_x$ of $(0, \dots, 0, -|x-z_x|)$ that is a diffeomorphism onto its image. From now on, we assume $\eps$ to be small enough such that $B_\eps^{-|x-z_x|} \subset V_x$, where $B_\eps^{-|x-z_x|} = \{h \in B_\eps(0) \mid h_d < -|x-z_x| \}$ as in the notation introduced above.
The inverse of the transformation $\psi_x$ is given by $(h', h_d) \mapsto (h', h_d+\gamma(h') + |x - z_x|)$. 
Since $\psi_x \big (0, \dots, 0, -|x-z_x|) \big) = (0, \dots, 0, -|x-z_x|)$, $\psi_x$ straightens the boundary to a hyperplane orthogonal to $e_d$ through $(0, \dots, 0, -|x-z_x|)$, which leaves us to integrate over $B_\eps^{-|x-z_x|}$. 
Observing that $\det(D \psi_x^{-1}) = 1$, we obtain the fifth equality.

Since each transformation $\psi_x^{-1}$ is continuous and bounded on $B_\eps$ independently of $x$, taking it into account only produces an error term that vanishes as $\eps \to 0$. Thus we may and do refrain from explicitly mentioning $\psi_x^{-1}$ in the following calculations. 
Therefore, after rotating and straightening, we have, ignoring the constants $a_d$ and $\eps^{-d-2}$:
\begin{alignat}{1}
    &-2 \int_{Q^+_\eps} \int_{B_\eps^{-|x - z_x|}} 
        \int_0^1 \int_0^1 F(x + t R^T_x h) \cdot R^T_x h  \cdot R^T_x h 
        \,\intd t \intd h \intd x\nonumber\\ 
	=&- \int_{Q^+_\eps} \int_{B_\eps^{-|x - z_x|}}
	    \int_0^1 F\big(x + tR^T_x(h', h_d)\big)\cdot R^T_x(h', h_d)
	    \,\intd t \intd h \intd x\nonumber\\
    &-  \int_{Q^+_\eps} \int_{B_\eps^{-|x - z_x|}}
	    \int_0^1 F \big(x+tR^T_x(-h', h_d)\big)\cdot R^T_x(- h', h_d)
	    \,\intd t \intd h \intd x\nonumber\\
	=&- \int_{Q^+_\eps} \int_{B_\eps^{-|x - z_x|}}
	    \int_0^1 [F\big(x + t R^T_x(h', h_d)\big)-F(x)]\cdot R^T_x(h', h_d)
	    \,\intd t \intd h \intd x\label{eqn:term1}\\
    & - \int_{Q^+_\eps} \int_{B_\eps^{-|x - z_x|}}
	    \int_0^1 [F \big(x+t R^T_x(-h', h_d)\big)-F(x)]\cdot R^T_x(- h', h_d)
	    \,\intd t \intd h \intd x\label{eqn:term2}\\
	& - \int_{Q^+_\eps} \int_{B_\eps^{-|x - z_x|}}
	    F(x)\cdot R^T_x(h', h_d) + F(x)\cdot R^T_x(- h', h_d)
	    \,\intd h \intd x\label{eqn:term3}
\end{alignat}
Here, due to the straightening, we were able to use the symmetry of the integral in the variables $h_1, \dots, h_{d-1}$. 

The first two terms can be treated in the following way. 
The fundamental theorem of calculus yields, for arbitrary $h \in \Rd$, $[F\big(x + th)\big) - F(x)]\cdot h = \int_0^1 [DF(x+sth)h]\cdot th \intd s$. 
Since $F \in C_b^1(\Rd; \Rd)$, we have $|DF(h)| \leq ||F||_{C_b^1(\Rd)}$, again for any $h$.
Using these two observations and putting absolute values, we estimate (\ref{eqn:term1}). Uneventful positive constants are collected under $c$, which may change from line to line. 
We obtain
\begin{alignat*}{1}
    c \: \eps^{-d-2} \int_{Q^+_\eps} \int_{B_\eps^{-|x-z_x|}} 
        &\int_0^1 \int_0^1 [DF(x+ st R^T_x h) R^T_x h ] \cdot t R^T_x h 
        \,\intd s \intd t \intd h \intd x\\
    \leq &\; c \: \eps^{-d-2} ||F||_{C_b^1(\Rd)}
        \int_{(-\rho, \rho)^{d-1}}\int_{\gamma(x')}^{\gamma(x')+\eps} \int_{B_\eps} |h|^2 
        \,\intd h \intd x_d \intd x'\\
    = &\; c \; \eps^{-d-2} ||F||_{C_b^1(\Rd)} 
        \int_{(\rho, \rho)^{d-1}} \int_{\gamma(x')}^{\gamma(x')+\eps} \eps^{d+2}
        \,\intd x_d \intd x'\\
    = &\; c \; ||F||_{C_b^1(\Rd)} 
    		\int_{(\rho, \rho)^{d-1}} \eps \intd x' = \; c \; \eps.
\end{alignat*}
This vanishes for $\eps \to 0$. A similar calculation shows that (\ref{eqn:term2}) goes to zero as well. 

Now we consider (\ref{eqn:term3}), which will produce the inner product of the local vector field $F$ with the outward unit normal vector. 
Observe, first,
\begin{alignat*}{1}
	&  F(x) \cdot R^T(h', h_d) +  F(x) \cdot R^T(-h', h_d) \\
	=& \sum_{j=1}^{d} \sum_{i=1}^{d-1} r_{i, j} h_i F_j(x) 
			+ \sum_{j=1}^d r_{d, j} h_d F_j(x)
	- \sum_{j=1}^{d} \sum_{i=1}^{d-1} r_{i, j} h_i F_j(x) 
			+ \sum_{j=1}^d r_{d, j} h_d F_j(x)\\
	=& \; 2 \Big(\sum_{j=1}^d r_{d, j} F_j(x)\Big) h_d
	= 2 \big(F(x) \cdot \nv(z_x)\big) h_d.
\end{alignat*}
Moreover, for any $h \in B_\eps(0)$, we know that $h$ is, by definition, in $B_\eps^{-|x - z_x|}$ if and only if $h_d < -|x - z_x|$. This is equivalent to 
\[\frac{|x_d - \gamma(x')|} {|x-z_x|} \; h_d < -|x_d - \gamma(x')|.\]

We now bring these two observations to bear on the third term, (\ref{eqn:term3}): 
\begin{alignat*}{1}
    &- 2 a_d \eps^{-d-2} 
        \int_{(-\rho, \rho)^{d-1}} \int_{\gamma(x')}^{\gamma(x')+\eps} \int_{B_\eps^{-|x - z_x|}}
	    \big(F(x) \cdot \nv(z_x)\big) h_d
	    \,\intd h \intd x_d \intd x'\\
	=&- 2 a_d \eps^{-d-2}  
		\int_{(-\rho, \rho)^{d-1}} \int_{\gamma(y')}^{\gamma(y')+\eps} 
		\big(F(x) \cdot \nv(z_x)\big)
		\frac {|x_d - \gamma(x')|}{|x - z_x|} 
		\int_{B_\eps^{-|x_d - \gamma(x')|}} h_d 
		\,\intd h \intd x_d \intd x'\\
	=&  \int_{(-\rho, \rho)^{d-1}}  
		\int_{0}^{\eps} 
		\big( F(x) \cdot \nv(z_x)\big)
		\frac {|w_d|}{|x - z_x|}
		\Big(- 2 a_d \eps^{-d-2} \; \int_{B_\eps^{-|w_d|}} h_d \,\intd h\Big) \intd w_d \intd x'
\end{alignat*}

Note that we did not explicitly write out the substitution $w_d \defeq x_d - \gamma(x')$ occurring in the variable $x$, which should of course be read as $x = (x', w_d + \gamma(x'))$. 

In \cite[Lemma 4.1]{GH22} it is proved that under conditions like ours, we have 
\[
	\frac {|x_d - \gamma(x')|}{|x - z_x|} \to \sqrt{ 1 + \;|\nabla \gamma(x')|^2} 
    \textrm{ as } x_d \to \gamma(x').
\]
This is the final ingredient we will need to wrap up the proof. 
Take 
\[
	g(x', w_d + \gamma(x')) \defeq 
	F(x', w_d + \gamma(x')) \cdot \nv(z_{(x', w_d + \gamma(x'))})
	\frac {|w_d|}{|(x', w_d + \gamma(x')) - z_{(x', w_d + \gamma(x'))}|}
\]
for a fixed $x' \in (-\rho, \rho)^{d-1}$ as a function in $w_d$.
For a sequence $(x', x_d^{(k)})_{k \in \N}$ of points in $\Omega^c$ such that $x_d^{(k)} \to \gamma(x')$, we know that any sequence of minimizers of their distance to the boundary, $(z_{(x', x_d^{(k)})})_{k \in \N}$, converges to $z_{(x', \gamma(x'))} = (x', \gamma(x')) \in \partial\Omega$ as $k \to \infty.$
Since $F$ and the normal vector $\nv(z_x)$ are continuous, we thus obtain 
    $$g(x', w_d + \gamma(x')) 
    \to F(x', \gamma(x')) \cdot \nv(x', \gamma(x')) \sqrt{ 1 + \;|\nabla \gamma(x')|^2}$$
as $w_d \to 0$. With the approximate identity 
$k_\eps(w_d) = - 2 a_d \eps^{-d-2} \; \int_{B_\eps^{-|w_d|}} h_d \intd h \; \mathbbm{1}_{(0, \eps)}(w_d)$
established in Lemma \ref{lemma:approx_id} and the fact that $g$ is bounded and continuous, we see that the limit for $\eps \to 0^+$ is precisely 
\begin{alignat*}{1}
	\int_{(-\rho, \rho)^{d-1}}& \lim_{\eps \to 0^+} 
        \int_\R k_\eps(w_d) g(x', w_d+\gamma(x')) \intd w_d \intd x' \\
	&= \int_{(-\rho, \rho)^{d-1}} g(x', \gamma(x')) \intd x' \\
	&= \int_{(-\rho, \rho)^{d-1}} 
	F(x', \gamma(x')) \cdot \nv(z_{(x', \gamma(x'))})
    \sqrt{ 1 + |\nabla \gamma(x')|^2 } \,\intd x'\\
	&= \int_{Q \cap \partial\Omega} 
	F(x) \cdot \nv(x)
	\intd\sigma_{d-1}(x).
\end{alignat*}	
\end{proof}

We have thus completed the proof of \eqref{eq:div-theo}.

\medskip

We now explain how to relax the assumptions on $\Omega$ in \autoref{thm:normal_conv}. We say that a bounded domain $\Omega$ is a \textit{piecewise $C^1$-domain} if its boundary is Lipschitz-continuous and the boundary $\partial \Omega$ can be decomposed into $N+1$ disjoint sets                    
	\begin{align*}
		\partial \Omega = G_1 \cup \ldots \cup G_N \cup B\,, 	 
	\end{align*}
such that $\mathcal H^{d-1}(B) = 0$, each set $G_i$ is connected, and the outward unit normal vector $\vec{n}|_{G_i}$ is uniformly continuous. 
For our proof, we need to impose a further regularity condition on $\partial \Omega$. Given a bounded domain $\Omega$, a set $M \subset \partial \Omega$ with outward unit normal vector $\vec{n}$, and $\eps > 0$, we define the \textit{exterior $\eps$-neighborhood of $M$} by 
\begin{align*}
M^{\eps} &= \{ x \in \Omega^c \, | \, x = z + t \vec{n}(z) \text{ for some } z \in  \partial \Omega, t \in (0, \eps) \} \,.
\end{align*}

\begin{defn}
	Let $\Omega$ be a piecewise $C^1$-domain in the above sense. We say that $\partial \Omega$ is \textit{accessible from the exterior} if for every $i \in \{1, \ldots, N\}$ and every $\delta > 0$, there is a set $G_{i, \delta} \subset G_i$ and $\eps_0 > 0$ with
	\begin{align*}
		\forall \; \eta \leq \delta  \;: \;  G_{i, \delta}  \subset G_{i, \eta} \text{ and }
		\forall \; i \ne j, \eps \in (0, \eps_0) \; : \; G_{i, \delta}^\eps \cap G_{j, \delta}^\eps = \emptyset\,,
	\end{align*}
	such that
	\begin{align}\label{eq:def-B-delta-eps}
		\bigcup_{\delta > 0}  G_{i, \delta} = G_i \text{ and }
		\lim\limits_{\delta \to 0} \lim\limits_{\eps \to 0}   \eps^{-1} \left| B(\delta,\eps) \right|  = 0\,,
	\end{align}
where $B(\delta,\eps) := \Omega^c_\eps \setminus \big(\bigcup_{i=1}^N  G_{i, \delta}^\eps \big)$\,.   
\end{defn}

For every $i$, the sets $G_{i, \delta}$ exhaust $G_i$ as $\delta \to 0$, and the neighborhoods $G_{j, \delta}^\eps$ and $G_{i, \delta}^\eps$ are disjoint if $\eps$ is sufficiently small. In this way, one can approximate the sets $G_i$ from the exterior $\Omega^c$ in a well-defined way.

Let us discuss some examples.

\begin{example}
(a) A cube in $\R^2$ is a piecewise $C^1$-domain, whose boundary is accessible from the exterior. For $\Omega = (-1,1) \times (-1,1) \subset \R^2$, one can choose $B$ as the set of the four corners and the sets $G_i$ as the four edges. Since $\Omega$ is convex, we can choose $G_{i, \delta} = G_i$ for every $\delta$. The value of $\eps_0$ can be chosen arbitrarily.\\
(b) For non-convex domains, the strength of the above definition shows. Let $\Omega = (-1,1) \times (-1,1)  \setminus \overline{(0,1) \times (0,1)}$ be an $L$-shaped domain in $\R^2$. All but two edges can be dealt with as in the case of the cube. The two edges $G_1$ connecting the points $(0,0)$ and $(0,1)$ resp. $G_2$ connecting  $(0,0)$ and $(1,0)$ are approximated by  $G_{1, \delta} = \{(0,x) | \, \delta < x < 1 \}$ resp. $G_{2, \delta} = \{(x,0) | \, \delta < x < 1 \}$ and $\eps_0$ can be chosen as $\delta/2$.\\
(c) Every compact simply connected set in $\R^3$ whose boundary is a simple closed polygon and whose interior is a Lipschitz domain is a piecewise $C^1$-domain whose boundary is accessible from the exterior.
\end{example}

Let us show that \autoref{thm:normal_conv} holds true for domains whose boundary is accessible from the exterior. 
The main idea is to decompose $\Omega^c_\eps$ for sufficiently small $\eps$ as follows. For every $\delta > 0$ there is $\eps_0 \in (0, 1)$ and a disjoint decomposition of the form
\begin{align*}
 \Omega^c_\eps = G_{1, \delta}^\eps \cup \ldots \cup G_{N, \delta}^\eps \cup B(\delta, \eps) \,,
\end{align*}
where $\eps \in (0, \eps_0)$.

Applying our proof of \autoref{thm:normal_conv}, we obtain for every $\delta > 0$,
\begin{align*}
	\lim\limits_{\eps \to 0} \int_{\Omega^c} \mathcal{N}_\eps f_\eps (x) \intd x =  \sum_{i=1}^N \int_{G_{i, \delta}} F(x) \cdot \vec{n}(x) \intd \sigma_{d-1} (x) +  \lim\limits_{\eps \to 0} \int_{B(\eps, \delta)} \mathcal{N}_\eps f_\eps (x) \intd x \,.
\end{align*}
The second term on the right-hand side converges to $0$ as $\delta \to 0$ due to property \eqref{eq:def-B-delta-eps} in the above definition. Considering the limit $\delta \to 0$, we obtain 
\begin{align*}
		\lim\limits_{\eps \to 0} \int_{\Omega^c} \mathcal{N}_\eps f_\eps (x) \intd x =  \sum_{i=1}^N \int_{G_{i}} F(x) \cdot \vec{n}(x) \intd \sigma_{d-1} (x) =  \int_{\partial \Omega} F(x) \cdot \vec{n}(x) \intd \sigma_{d-1} (x) \,.
\end{align*}	

This completes the proof of \autoref{thm:normal_conv} for domains whose boundary is accessible from the exterior.

\section{Fractional operators, perimeters, and the mean curvature revisited}\label{sec:frac-operators} 

In this section, we explain how one can use our notion of fractional divergence in \eqref{eq:def-frac-div-grad} to represent the fractional $p$-Laplace, the fractional perimeter, and the fractional mean curvature. For each of these concepts, the local definition is provided together with the fractional variant that is prevalent in the literature. Rewriting the fractional definitions using our concept of fractional divergence exhibits structural similarities between the fractional and local concepts. We hope that this provides independent motivation for our choices in the definitions \eqref{eq:case-frac} and \eqref{eq:def-frac-div-grad}.  

\subsection*{The fractional $p$-Laplace} 

For $s \in (0, 1)$ and $p \in [1, \infty)$, the fractional $p$-Laplace of a sufficiently regular function $u: \R^d \to \R$ is often defined via
\begin{align*}
-(-\Delta)^s_p u (x) &= (1-s) \operatorname{pv.} \int\limits_{\R^d} |u(y)-u(x)|^{p-2} \big(u(y) - u(x) \big) \frac{\intd y}{|y-x|^{d+sp}} \,.
\end{align*}
With our definition of the fractional divergence $\operatorname{div}^{(s)}$ in \eqref{eq:def-frac-div-grad}, we obtain
\begin{align*}
-(-\Delta)^s_p u (x) &= (1-s) \operatorname{pv.} \int\limits_{\R^d} \big(\tfrac{|u(y)-u(x)|}{|y-x|^s} \big)^{p-2} \tfrac{u(y) - u(x)}{|y-x|^s} \tfrac{\intd y}{|y-x|^{d+s}} \\
&= (1-s) \left( \frac{\mathcal H^{d-1}(\mathbb S^{d-1})}{4 d (1-s) s} \right) \operatorname{div}^{(s)} \big(|\nabla^{(s)} u|^{p-2} \nabla^{(s)} u \big) (x) \\
&= \frac{\mathcal H^{d-1}(\mathbb S^{d-1})}{4 d s}  \operatorname{div}^{(s)} \big(|\nabla^{(s)} u|^{p-2} \nabla^{(s)} u \big) (x) \,. 
\end{align*} 
Using our fractional divergence and gradient, it is thus possible to write the fractional $p$-Laplace like the local $p$-Laplace $\Delta_p u = \operatorname{div} \cdot (|\nabla u|^{p-2} \nabla u)$. Similar representations have been given in other works such as \cite{MaSch18, Sil20, ByKi23}.

If, in the case $p=2$, one uses the standard definition of the fractional Laplace operator
\begin{align*}
- (-\Delta)^s u (x) &= c_{d,s} \operatorname{pv.} \int\limits_{\R^d} \frac{u(y) - u(x)}{|y-x|^{d+2s}} \intd y 
\end{align*} 
with a constant $c_{d,s}$ that guarantees $\mathcal{F}\big((-\Delta)^s u\big) (\xi) = |\xi|^{2s} \mathcal{F}(u) (\xi)$, then one obtains 
\begin{align}\label{eq:frac-laplace-new}
(-\Delta)^s u (x) &=  \frac{c_{d,s} \mathcal H^{d-1}(\mathbb S^{d-1})}{4 d (1-s) s}  \operatorname{div}^{(s)} \nabla^{(s)} u (x) \,.
\end{align}
Then $\frac{c_{d,s} \mathcal H^{d-1}(\mathbb S^{d-1})}{4 d (1-s) s} \to 1 $ as $s \to 1$, a proof of which can be found in several sources, see e.g. \cite{Fog20} for an elegant proof. This asymptotic behavior of the normalizing constants further illustrates the naturalness of our definitions. 

\subsection*{The fractional perimeter} \
The perimeter of a Lebesgue-measurable set $E \subset \Rd$ has been defined by De Giorgi in the following way:
\begin{align}\label{def:classical-perimeter}
P(E) = \sup \Big\{ \int_E \operatorname{div} \varphi (x) \intd x \, \big| \, \varphi \in C^\infty_c(\R^d; \R^d), |\varphi| \leq 1 \Big\}
\end{align}
The notion of the fractional perimeter, which was introduced first in \cite{CRS10}, is given as:
\begin{align*}
\operatorname{Per}_s(E) = \frac 1 {|B^{d-1}|} \int_E \int_{E^c} \frac{\intd x \intd y}{|y-x|^{d+s}} \,, 
\end{align*}
where $|B^{d-1}|$ is the $(d-1)$-dimensional Lebesgue measure of the unit ball in $\R^{d-1}$. It is a well known fact that 
\begin{align}\label{eq:perimeter-conv}
(1-s)\operatorname{Per}_s(E) \to \operatorname{Per}(E) \text{ as } s \to 1-\, ,
\end{align}
for sets $E \subset \Rd$ of finite perimeter and finite Lebesgue measure. See \cite[Theorem 1]{CV2011}, \cite[Theorem 2]{ADPM11} for a proof of the general case, for the global case discussed here, see \cite{BBM01} and specifically \cite[Theorem 1]{Dav02}. 

Using the nonlocal divergence, we can provide an alternative definition of the fractional perimeter as in \cite[Definition 3.1]{AJS22} that is analogous to the classical one: 
\begin{align*}
P_s(E) \defeq \sup \Big\{ c_{d, s} \int_E \mathcal{D}_s f (x) \intd x \, 
\big| \, &f \in C^1(\Rtd) \text{ antisymmetric, }
|f| \leq 1 \Big\}
\end{align*}
Note that the conditions on the antisymmetric function $f$ are modeled after the definition of the classical perimeter stated in \eqref{def:classical-perimeter}. The constant $c_{d, s}$ provides the correct scaling. Given \eqref{eq:case-frac}, it is defined as 
\begin{align}\label{eq:perimeter-constant}
c_{d, s} \defeq \frac {\mathcal H^{d-1}(\mathbb S^{d-1})}{4ds |B^{d-1}|} = \frac{\Gamma(d/2 + 1/2)\sqrt{\pi}}{4s \Gamma(d/2 + 1)} \, ,
\end{align}
where $\Gamma$ is Euler's gamma function. 

We prove that the two quantities coincide, specifically we show: 
\begin{prop}\label{prop:perimeter-equivalence}
	Let $E \subset \Rd$ be open and bounded. There holds
	\begin{align}\label{eq:perimeter-equivalence}
	(1-s)\operatorname{Per}_s(E) = P_s(E) \text{ for any } s\in(0,1).
	\end{align}
\end{prop}
This equivalence then immediately yields
\begin{align*}
P_s(E) \to \operatorname{Per}(E) \text{ as } s \to 1- \, ,
\end{align*}
for open and bounded sets of finite perimeter and Lebesgue measure, by \eqref{eq:perimeter-conv}.

\begin{proof}
	Observe first that, for any antisymmetric function $f: \Rtd \to \R$ with $|f| \leq 1$, there holds
	\begin{align*}
	c_{d, s} \int_E &\operatorname{div}^{(s)} f(x, y) \intd x \\
	&= \frac {\mathcal H^{d-1}(\mathbb S^{d-1})}{4ds |B^{d-1}|} \frac{2ds(1-s)}{\mathcal{H}^{d-1}(\mathbb{S}^{d-1})} 2 \int_E \int_{E^c} f(x, y) |y-x|^{-d-s} \intd y \intd x
	\leq (1-s)\operatorname{Per}_{s} (E),
	\end{align*}
	so $P_s(E) \leq (1-s)\operatorname{Per}_s(E)$.
	
	The converse inequality needs a bit more work. 
	Write $E_\eps \defeq \{ x \in E \mid \dist (x, \partial E) > \eps \}$ and $E^c_\eps \defeq \{ y \in E^c \mid \dist (y, \partial E) > \eps \}$.
	The second one is, perhaps, a slight abuse of notation. But the meaning in this proof will be unambiguous. 
	For any $s \in (0, 1)$, we construct a sequence of functions maximising $P_s(E)$. Obviously, a maximising function would be 
	\begin{align*}
	g(x, y) \defeq \mathbbm{1}_E (x) \mathbbm{1}_{E^c}(y) 
	- \mathbbm{1}_{E^c}(x) \mathbbm{1}_{E}(y),
	\end{align*}
	but this function is not permitted in the supremum that appears in the definition of $P_s(E)$. Instead, we approximate $g$ using a sequence of standard mollifiers $\varphi_\eps \in C_c^\infty (\Rd)$, i.e. functions satisfying 
	$\text{supp } \varphi_\eps \subset B_\eps(0)$,
	$0 \leq \varphi_\eps \leq 1$, and
	$\int_\Rd \varphi_\eps \intd x = 1$. 
	Moreover, we may assume $|\nabla \varphi_\eps | \leq c/\eps^{d+1}$ for some appropriate constant $c > 0$. 
	
	Now for any $\eps \in (0, 1)$, define
	\begin{align*}
	g_\eps (x, y) \defeq (\varphi_\eps \ast \mathbbm{1}_{E_\eps}) (x) (\varphi_\eps \ast \mathbbm{1}_{E^c_\eps})(y) 
	- (\varphi_\eps \ast \mathbbm{1}_{E^c_\eps})(x) (\varphi_\eps \ast \mathbbm{1}_{E_\eps})(y).
	\end{align*}
	For each $\eps \in (0,1)$, $g_\eps$ is smooth as well as antisymmetric, and $|g_\eps| \leq 1$. Thus, each $g_\eps$ can be used in the supremum that defines $P_s(E)$. 
	Given this setup, we consider the following integral:  
	\begin{align*}
	c_{d, s} \int_E \operatorname{div}^{(s)} g_\eps(x) \intd x
	= & \frac{(1-s)}{|B^{d-1}|} \int_E \int_{E^c} (\varphi_\eps \ast \mathbbm{1}_{E_\eps}) (x) (\varphi_\eps \ast \mathbbm{1}_{E^c_\eps})(y) |y - x|^{-d-s} \intd y \intd x
	\end{align*}
	We split the integral over $E \times E^c$ into four parts,  
	\begin{align}\label{eq:perimeter-split}
	(E_{2\eps} \times E^c_{2\eps})
	\cup (E_{2\eps} \times E^c \setminus E^c_{2\eps})
	\cup (E \setminus E_{2\eps} \times E^c_{2\eps})
	\cup (E\setminus E_{2\eps} \times E^c \setminus E^c_{2\eps}),
	\end{align}
	which we treat separately. 
	
	The first set contains the good part. For $x \in E_{2\eps}$, the convolution evaluates to 
	\begin{align*}
	\int_{B_\eps(x)} \varphi_\eps(x - z) \mathbbm{1}_{E_\eps} (z) \intd z = 1 ,
	\end{align*}
	since $\mathbbm{1}_{E_\eps} (z) = 1$ for any $x \in E_{2\eps}$ and $z \in B_\eps(x)$. The same applies for the convolution term involving $y \in E^c_{2\eps}$. 
	Hence 
	\begin{align*}
	&\frac{(1-s)}{|B^{d-1}|} \int_{E_{2\eps}} \int_{E_{2\eps}^c} (\varphi_\eps \ast \mathbbm{1}_{E_\eps}) (x) (\varphi_\eps \ast \mathbbm{1}_{E^c_\eps})(y) |y - x|^{-d-s} \intd y \intd x\\
	= &\frac{(1-s)}{|B^{d-1}|} \int_{E_{2\eps}} \int_{E_{2\eps}^c} |y - x|^{-d-s} \intd y \intd x,
	\end{align*}
	which converges to $(1-s) \operatorname{Per}_s(E)$ for $\eps \to 0$ by monotone convergence.
	
	We will now argue that the integrals over the remaining three terms in \eqref{eq:perimeter-split} all vanish. It is therefore no problem to ignore the constants in front of the integral, since they only depend on $s$ and $d$. We will omit them in the following arguments. The integrals over the middle two sets in \eqref{eq:perimeter-split} can be dealt with in the same way. They are both well-behaved because $x$ and $y$ cannot get arbitrarily close to each other. We can always estimate the convolution terms by 1, which we readily do, to obtain
	\begin{align*}
	&\int_{E\setminus E_{2\eps}} \int_{E_{2\eps}^c} (\varphi_\eps \ast \mathbbm{1}_{E_\eps}) (x) (\varphi_\eps \ast \mathbbm{1}_{E^c_\eps})(y) |y - x|^{-d-s} \intd y \intd x\\
	\leq & \int_{E\setminus E_{2\eps}} \int_{E_{2\eps}^c} |y - x|^{-d-s} \intd y \intd x.
	\end{align*}
	For any $x \in E\setminus E_{2\eps}$, there holds $E_{2\eps}^c \subset B_{2\eps}(x)^c$. This implies 
	\begin{align*}
	&\int_{E\setminus E_{2\eps}} \int_{E_{2\eps}^c} |y - x|^{-d-s} \intd y \intd x
	\leq \int_{E\setminus E_{2\eps}} \int_{B_{2\eps}(x)^c} |y-x|^{-d-s} \intd y \intd x\\
	= & \int_{E\setminus E_{2\eps}} \int_{B_{2\eps}(0)^c} |h|^{-d-s} \intd h \intd x
	= \int_{E\setminus E_{2\eps}} \mathcal{H}^{d-1}(\Sd) \int_{2\eps}^\infty r^{-d-s+d-1} \intd r \intd x\\
	= & \; \frac{\mathcal{H}^{d-1}(\Sd)}{s} \int_{E\setminus E_{2\eps}} (2\eps)^{-s} \intd x
	\leq c \eps^{1-s}.
	\end{align*}
	For the last inequality, we have used that $E$ is bounded, whence $E \setminus E_{2\eps}$ can be estimated by some constant that scales like $\eps$. Since $s \in (0, 1)$, we have $1-s > 0$, so the whole term vanishes for $\eps \to 0$. 
	
	The integral over the third set in \eqref{eq:perimeter-split} follows in the same manner with an application of Fubini-Tonelli.  
	
	Now we calculate the integral over the last set in \eqref{eq:perimeter-split}. The delicate part is the one where $x$ and $y$ are close together, so we split the domain of integration of the inner integral once more. For each $x \in E\setminus E_{2\eps}$, we consider $(E^c \setminus E^c_{2\eps}) \cap B_{2\eps}(x)^c$ and $(E^c \setminus E^c_{2\eps}) \cap B_{2\eps}(x)$. 
	In the first domain of integration, we can proceed as in the two foregoing cases to show that the term vanishes as $\eps \to 0$. 
	
	Before we evaluate the second domain of integration, we make some observations. For $y \in E^c \setminus E^c_{2\eps}$, we have 
	$\int_{B_\eps (x)} \varphi_\eps (y-z) \mathbbm{1}_{E_\eps}(z) \intd z = 0$.
	Indeed, since $\mathbbm{1}_{E_\eps}(z) = 1$ iff $z \in E_\eps$ and $y \in E^c$, we have that $|y - z| > \eps$, so $\varphi_\eps(y-z) = 0$. The analogous case for  $(\varphi_\eps \ast \mathbbm{1}_{E^c_\eps})(y)$  naturally holds as well. 
	
	With first-order Taylor expansion, we can thus estimate the convolution terms as follows:
	\begin{align*}
	\int_{B_\eps(x)} \varphi_\eps (x-z) \mathbbm 1_{E_\eps}(z) \intd z
	&= \int_{B_\eps(x)} \varphi_\eps (x-z) - \varphi_\eps(y-z) \mathbbm 1_{E_\eps}(z) \intd z \\
	&\leq \int_{B_\eps(x)} \nabla \varphi_\eps (z_x) \cdot (x - z - y + z) \intd z \\
	&\leq \int_{B_\eps(x)} |\nabla \varphi_\eps (z_x)| |x - y| \intd z \\
	&\leq \eps^d |B_1| \frac c {\eps^{d+1}} |y-x| = \frac c \eps |y-x|
	\end{align*}
	In the last line, we have employed the fact that $|\nabla \varphi_\eps| \leq c/\eps^{d+1}$. One obtains analogously
	\begin{align*}
	\int_{B_\eps(y)} \varphi_\eps (y-z) \mathbbm 1_{E^c_\eps}(z) \intd z 
	\leq \frac c \eps |y-x|. 
	\end{align*}
	We use these estimates to evaluate the integral over the last remaining domain of integration:
	\begin{align*}
	&\int_{E \setminus E_{2\eps}} \int_{(E^c \setminus E^c_{2\eps}) \cap B_{2\eps}(x)}
	(\varphi_\eps \ast \mathbbm{1}_{E_\eps}) (x) (\varphi_\eps \ast \mathbbm{1}_{E^c_\eps})(y) |y - x|^{-d-s} \intd y \intd x\\
	\leq & \frac c {\eps^2} \int_{E \setminus E_{2\eps}} \int_{(E^c \setminus E^c_{2\eps}) \cap B_{2\eps}(x)} |y-x|^{2-d-s} \intd y \intd x\\
	\leq & \frac c {\eps^2} \int_{E \setminus E_{2\eps}} \mathcal H^{d-1}(\mathbb S^{d-1}) \int_0^{2\eps} r^{1-s} \intd r \intd x\\
	\leq & \frac c {\eps^2} \int_{E \setminus E_{2\eps}} \frac {\mathcal H^{d-1}(\mathbb S^{d-1})} {(2-s)} (2\eps)^{2-s} \intd x
	= \frac c {\eps^2} c \eps (2\eps)^{2-s} = c \eps^{1-s}
	\end{align*}
	Note that we have used again the boundedness of $E$ to estimate $|E \setminus E_{2\eps}|$ by $c \eps$ for some appropriate constant $c$. 
	Putting everything together yields \eqref{eq:perimeter-equivalence}. 
\end{proof}

\subsection*{The fractional mean curvature} 
Classically, several equivalent definitions have been put forward for the mean curvature of a set $E \subset \Rd$ with sufficiently regular boundary. One prominent way is to define it as the divergence of the unit normal vector field at the boundary.
In the nonlocal setting, the fractional (or nonlocal) mean curvature has been introduced as the first variation of the fractional perimeter, see \cite{CRS10}. Given $x \in \partial E$, the fractional mean curvature is defined for $s \in (0, 1)$ by
\begin{align*}
H_s(x; E) = \frac{1}{|B^{d-1}|} \operatorname{pv.} \int_{\R^d} \big( \mathbbm{1}_{E^c}(y) - \mathbbm{1}_{E}(y) \big) |y-x|^{-d-s} \intd y
\end{align*}
Note that this definition is well-posed, see \cite[Lemma 7]{AbVa14} and \cite[Cor. 3.5]{Coz15}. Like the fractional perimeter, the fractional mean curvature, too, converges to the classical mean curvature as $s \to 1-$, see \cite[Lemma 9]{CV2013} or \cite[Theorem 12]{AbVa14}. If $H(x; E)$ is the classical mean curvature at $x \in \partial E$ for $\partial E$ sufficiently regular, there holds
\begin{align*}
(1-s)H_s(x; E) \to H(x; E) \text{ as } s \to 1-\,.
\end{align*}

We can re-write $H_s(x; E)$ with the help of the fractional divergence $\operatorname{div}^{(s)}$ defined in \eqref{eq:def-frac-div-grad}. To this end, let us define a nonlocal version of the normal vector field $n:\R^d \times \R^d \to \{-1,0,1\}$ with respect to the set $E$ as follows. Given $x \in \R^d$, we denote by $\delta_x \in \R$ the directed distance to the boundary, specifically
\[
\delta_x= \\
\begin{cases}
\dist (x, \partial E) & \text{ for } x \in E^c \,, \\
- \dist (x, \partial E) & \text{ for } x \in E \,. 
\end{cases}
\]
Next, we set $n:\R^d \times \R^d \to \{-1,0,1\}$ by
\[
n(x,y) = \\
\begin{cases}
1 & \text{ if } \delta_y > \delta _ x \,, \\
0 & \text{ if } \delta_y = \delta _ x \,, \\
-1 & \text{ if } \delta_y < \delta _ x \,.
\end{cases}
\]
Then, using the constant $c_{d, s}$ defined in \eqref{eq:perimeter-constant} that was also used for the fractional perimeter, one observes for $x \in \partial E$
\begin{align*}
(1-s) H_s(x; E) = c_{d, s} \operatorname{div}^{(s)} n  (x) \,.
\end{align*}
Thus the nonlocal mean curvature emerges again as the divergence of the normal vector field as in the local case. As before, the equivalence of the two fractional concepts yields the convergence of $\operatorname{div}^{(s)} n  (x)$ to the classical mean curvature.



\end{document}